\documentclass[12pt,oneside]{amsart}
\usepackage{amsaddr}
\usepackage[utf8]{inputenc}
\usepackage[T1]{fontenc}
\usepackage{lmodern}
\usepackage{amssymb}
\usepackage[bottom=37mm]{geometry}
\usepackage{amsmath}
\usepackage{amsthm}
\usepackage{graphicx}
\usepackage{color}
\usepackage[onehalfspacing]{setspace}
\usepackage{verbatim}
\usepackage[hidelinks]{hyperref}
\usepackage{lipsum}

\newcommand{\IR}{\ensuremath{\mathbb{R}}}
\newcommand{\IN}{\ensuremath{\mathbb{N}}}

\newcommand{\IE}{\ensuremath{\mathbb{E}}}

\newcommand{\EE}{\mathbb{E}}
\newcommand{\NN}{\mathbb{N}}
\newcommand{\PP}{\mathbb{P}}
\newcommand{\RR}{\mathbb{R}}
\newcommand{\R}{\mathbb{R}}

\renewcommand{\rho}{\varrho}

\newcommand{\norm}[1]{\left\Vert#1\right\Vert}
\newcommand{\Bnorm}[1]{\Big\Vert#1\Big\Vert}

\newcommand{\abs}[1]{\left|#1\right|}

\DeclareMathOperator{\dist}{dist}

\DeclareMathOperator{\vol}{vol}
\DeclareMathOperator{\inj}{inj}

\newcommand{\dd}{\mathrm{d}}
\newcommand{\dvol}{\mathrm{dvol}}
\newcommand{\bfone}{\mathbf{1}}

\newcommand{\ls}{\lesssim}
\newcommand{\gs}{\gtrsim}

\makeatletter
\@namedef{subjclassname@2020}{\textup{2020} Mathematics Subject Classification}
\makeatother

\newtheorem{thm}{Theorem}
\newtheorem{cor}{Corollary}

\theoremstyle{plain}
\newtheorem{lemma}{Lemma}
\newtheorem{prop}{Proposition}
\theoremstyle{definition}

\newtheorem{rem}{Remark}

\title[Function recovery on manifolds]{Function recovery on manifolds\\ using scattered data}

\author[DK]{David Krieg and Mathias Sonnleitner}
\address[]{Institut f\"ur Analysis, 
Johannes Kepler Universit\"at Linz,\\ Altenbergerstrasse 69, 4040 Linz, Austria}
\curraddr[D. Krieg]{Faculty of Computer Science and Mathematics, 
University of Passau, Innstrasse 33, 94032 Passau, Germany}
\email{david.krieg@uni-passau.de} 
\curraddr[M. Sonnleitner]{Institute of Mathematical Stochastics, 
University of M\"unster, Orl\'eans-Ring 10, 48149 Münster, Germany}
\email{mathias.sonnleitner@uni-muenster.de}


\date{\today}

\begin{document}

\begin{abstract}	
	We consider the task of recovering a Sobolev function on a connected compact Riemannian manifold $M$ 
when given a 
sample on a finite point set.
We prove that the quality of the sample is given by the $L_\gamma(M)$-average
of the geodesic distance to the point set and determine the value of $\gamma\in (0,\infty]$. 
This extends our findings on bounded convex domains [IMA~J.~Numer.~Anal., 44:1346--1371, 2024]. 
As a byproduct, we prove the optimal rate of convergence of the $n$-th minimal worst case error
for $L_q(M)$-approximation for all $1\le q \le \infty$.
 
Further, a limit theorem for moments of the average 
distance to a set consisting of i.i.d.\ uniform points is proven.
This yields that a random sample is asymptotically as good as 
an optimal sample 
in precisely those cases with $\gamma<\infty$.
In particular, we obtain that
cubature formulas with random nodes are asymptotically 
as good as optimal cubature formulas
if the weights are chosen correctly.
This closes a logarithmic gap left open by Ehler, Gr\"af and Oates [Stat.~Comput., 29:1203-1214, 2019].
\end{abstract}

\maketitle

Let $M$ be a $d$-dimensional connected compact Riemannian manifold $M$. We are interested in the problem of recovering a function $f\colon M\to \IR$ which is only known
through a finite number of function values,
obtained on a finite set of sampling points.
We study this problem in the worst case setting,
where the function is assumed to belong to the unit ball of a Bessel potential (Sobolev) space
\[
	H^{s}_{p}(M)
	\,:=\,
	\big\{f\in L_{p}(M)\colon \|f\|_{H^s_p(M)}:=\|(I-\Delta_M)^{s/2}f\|_{L_p(M)}<\infty\big\},
\]
where $I$ denotes the identity on $L_{p}(M)=L_p(M,\vol_M)$, $\vol_M$ the Riemannian volume measure and $\Delta_M$ the Laplace-Beltrami operator on $M$. We refer to Section~\ref{sec:sobolev} for more information about $H^s_p(M)$.

\medskip

We consider the recovery problem in $L_q(M)$ for $1\le q \le \infty$. For a recovery operator $A\colon H^{s}_{p}(M) \to L_q(M)$ the worst case error is defined by
\[
 e(A,H^s_p(M),L_q(M)) \,:=\, \sup_{\Vert f \Vert_{H^{s}_{p}(M)} \le 1} \Vert f - A(f) \Vert_{L_q(M)}.
\]
We consider recovery operators of the form
$A=\varphi \circ N_P$,
where $P \subset M$ is a given (finite) set of sampling points,
$N_P\colon H^s_p(M) \to \R^P$ is the information mapping 
of the form $N_P(f)=(f(x))_{x\in P}$
and $\varphi\colon \R^P \to L_q(M)$ is any reconstruction map.
The minimal worst case error 
that can be achieved with a sampling operator
of this form
is denoted by
\[
 e(P,H^{s}_{p}(M),L_q(M)) \,:=\, \inf_{A=\varphi \circ N_P} e(A,H^s_p(M),L_q(M)).
\]
It is a measure for the quality of 
the sampling point set $P$.
If we only allow linear
reconstruction maps $\varphi$, 
which might be beneficial for practical purposes,
the infimum is denoted by $e^{\rm lin}(P,H^{s}_{p}(M),L_q(M))$.

\medskip

We also consider the problem of integration with respect to the
Riemannian volume $\vol_M$.
The worst case error of an operator $A\colon H^s_p(M) \to \R$
for the integration problem on $H^s_p(M)$ 
is defined by
\[
e(A,H^s_p(M),{\rm INT}) \,:=\, \sup_{\Vert f \Vert_{H^s_p(M)} \le 1} \left| \int f\, \dvol_M - A(f) \right|.
\]
The minimal worst case error 
that can be achieved with a given set of sampling points $P\subset M$
is denoted by
\[
 e(P,H^s_p(M),{\rm INT}) \,:=\, \inf_{A=\varphi \circ N_P} e(A,H^s_p(M),{\rm INT})\,,
\]
where the infimum ranges over all mappings $\varphi \colon \R^P \to \R$.
Note that, in this case, the infimum will not change 
if we only minimize over all linear 
mappings $\varphi$.
This is a classical result due to Smolyak and Bakhvalov,
see, e.g.,~Theorem 4.7 in the book by Novak and Wo\'zniakowski \cite{NW08}.
Thus, it is no restriction to assume that $A$ is of the form
\[
 A(f) \,=\, \sum_{x\in P} w_x f(x),
\]
i.e., a cubature formula with nodes $x\in P$ and weights $w_x\in\R$.
In the case of $p=2$ the optimal weights can be computed using the reproducing kernel of $H^s_2(M)$, see for example Section 2.4 in \cite{KS20}, where an analogous computation is described in the Euclidean setting.  
\medskip

The following theorem, our main result, characterizes the quality of the sampling point set $P$ for recovery and integration in terms of its distance function defined by
\[
\dist_M(\,\cdot\,,P)\colon M\to [0,\infty), \quad \dist_M(x,P):=\inf_{y\in P}\,\dist_M(x,y),
\]
where $\dist_M(x,y)$ denotes the Riemannian distance between $x,y\in M$.

\begin{thm}\label{thm:main}
Let $M$ be a connected compact Riemannian manifold of dimension $d\in\NN$
and let $1< p< \infty$, $1\le q\le \infty$ and $s>d/p$.
Then we have for any nonempty and finite point set $P\subset M$ that
\begin{align}
	\tag{\textit a}
  &e\big(P,H^s_p(M), L_q(M)\big)\, \asymp \, e^{\rm lin}\big(P,H^s_p(M), L_q(M)\big) \, \asymp \,
  \big\|\dist_M(\,\cdot\,, P)\big\|_{L_{\gamma}(M)}^\alpha \\
	\tag{\textit b}
	&e\big(P,H^s_p(M), {\rm INT}\big)\, \asymp \, e\big(P,H^s_p(M), L_1(M)\big),
\end{align}
where the implied constants are independent of $P$ and
\[
 \gamma = \left\{ \begin{array}{ll} s(1/q-1/p)^{-1},  & q<p, \\ \infty, &  q\ge p, \end{array} \right.
 \qquad
 \alpha =  \left\{ \begin{array}{ll} s,  & q<p, \\ s-d(1/p-1/q), &  q\ge p. \end{array} \right.
\]
\end{thm}

Here, the notation $A(P)\asymp B(P)$ means that there exist implied constants $c,C>0$ such that $cA(P)\le B(P)\le CA(P)$ for all nonempty finite point sets $P\subset M$. Below, the notation is also used for a positive integer $n$ instead of $P$. 

\medskip

Theorem~\ref{thm:main} allows for a characterization of the quality of a set of sampling points $P$ in terms of an $L_\gamma(M)$-norm (or quasi-norm) of its distance function.  In the case $\gamma=\infty$, this is known as the covering radius of $P$ in $M$, which is the radius of the largest hole in the data set.  In the case $\gamma<\infty$, it is the radius of an average-sized hole in the data set and sometimes called the distortion of $P$ in $M$, see, e.g., Zador~\cite{Zad82}.

\medskip

In fact, Theorem~\ref{thm:main} is an extension of our previous result valid for bounded convex domains in Euclidean spaces, which reads as follows. 

\begin{thm}[{\cite[Theorem~0.1]{KS20}}]\label{thm:main-euclidean}
	Let $\Omega\subset \mathbb{R}^d$ be a bounded convex domain, $1\le p,q \le \infty$ and $s\in\mathbb{N}$ with $s>d/p$. Then the statement of Theorem~\ref{thm:main} remains true with the same $\gamma$ and $\alpha$ if $M$ is replaced by $\Omega$ and $H^s_p(M)$ is replaced by 
\[
W^s_p(\Omega):=\big\{f\in L_p(\Omega)\colon \|f\|_{W^s_p(\Omega)}:=\Big(\sum_{|\alpha|\leq s} \|D^{\alpha}f\|_{L_p(\Omega)}^p\Big)^{1/p}<\infty\big\},
\]
where $ \alpha\in \mathbb{N}_0^d $ is a multiindex, $ |\alpha|=\alpha_1+\ldots+\alpha_d $ and $D^\alpha f = \frac{\partial^{|\alpha|}}{\partial x_1^{\alpha_1}\cdots\partial x_d^{\alpha_d}} f$.
\end{thm}

In Euclidean space, this characterization was new in the case of $q<p$ and for integration and improved upon previous characterizations in terms of the covering radius given by Novak and Triebel~\cite{NT06} and Wendland \cite{Wen01,W04}.

\medskip

These results describe the quality of an arbitrary set of sampling points.
Let us now turn to the question of optimal sampling points.
It follows from elementary volume comparison arguments and Lemma~\ref{lem:ahlfors} below that, for all $0<\gamma\le \infty$ and $\alpha>0$, we have
\begin{equation}\label{eq:opt-dist}
 \inf_{P\subset M \colon |P| \le n} \norm{\dist_M(\,\cdot\,,P)}_{L_{\gamma}(M)}^\alpha \,\asymp\, n^{-\alpha/d},
\end{equation}
where the infimum is over all point sets of cardinality at most $n$. Note that exact asymptotics for this so-called minimal quantization error were obtained by Iacobelli \cite{Iac16} and Kloeckner \cite{Klo12}. In particular, Theorem~\ref{thm:main} implies the following behaviour for the sampling numbers, i.e., the sequence of minimal errors over all recovery operators using at most $n$ points.

\begin{cor}\label{cor:optimal}
Let $M$ be a connected compact Riemannian manifold of dimension $d\in\NN$
and let $1< p< \infty$, $1\le q\le \infty$ and $s>d/p$.
Then
\begin{align*}
&\inf_{P\subset M\colon |P|\le n}e^{\ast}\big(P,H^s_p(M), L_q(M)\big)\, \asymp \, n^{-s/d+(1/p-1/q)_+}\\
&\inf_{P\subset M\colon |P|\le n}e^{\ast}\big(P,H^s_p(M), {\rm INT}\big)\, \asymp \, n^{-s/d}
\end{align*}
where $e^\ast\in\{e, e^{\rm lin}\}$, the implied constants are independent of $n\in \IN$.
\end{cor}

These rates for the $n$-th minimal worst-case error are known for domains in Euclidean space, see Novak and Triebel \cite{NT06}.
It is not surprising that the same rates apply on compact Riemannian manifolds.
However, to the best knowledge of the authors,
Corollary~\ref{cor:optimal} appears to be the first result that actually provides all the rates in the Riemannian case,
even in the special case of the $d$-dimensional unit sphere $M=\mathbb{S}^d$.
In the following, we briefly survey the existing literature.
We also refer to \cite[Section~4.1]{LLG+24} for further references to the literature. 

\medskip

On the sphere, Hubbert and Morton~\cite{HM04} derive asymptotically optimal upper bounds using radial/zonal basis functions in the case of $p=2$ and $1\le q\le \infty$. In~\cite[Theorem~3.1]{Mha06} Mhaskar extends this using a similar approximation technique, but only in the case $p=q=\infty$ approximants are given in terms of function values. Building on H.~Wang and K.~Wang~\cite{WW16}, H.~Wang and Sloan~\cite{WS17} obtain Corollary~\ref{cor:optimal} in the case of $1\le p=q\le \infty$ using filtered hyperinterpolation. Improving upon Gr\"ochenig~\cite{Gro20}, Lu and H.~Wang~\cite{LW21} use a weighted least squares algorithm to obtain the optimal rates for $p=q=2$. 
 
On general manifolds, Mhaskar obtains in \cite[Theorem 3.1]{Mha10} an asymptotically optimal upper bound in the case $p=q$. In the general case, more is known about recovering only an integral instead of the function itself, see Brandolini et al.~\cite{BCC+14} and Ehler, Gräf and Oates~\cite{EGO19}.  We also note \cite[Collary~17]{KPU+23}, where a weighted least squares algorithm based on subsampled random points is used to obtain the upper bound of Corollary~\ref{cor:optimal} in the case $p=2$ and $1\le q\le\infty$. Very recently, after publication of our results, Li et al.~\cite{LLG+24}, obtain an extension of Corollary~\ref{cor:optimal} to $1\le p,q\le \infty$ and $s>d/p$ using weighted least $\ell_p$ approximation.

\medskip

The majority of the previously studied asymptotically optimal recovery algorithms rely on sequences of $n$-point sets $P_n=\{x_1^{(n)},\dots,x_n^{(n)}\}\subset M$ satisfying an Marcinkiewicz-Zygmund condition, i.e., the discrete and continuous $p$-norms are equivalent for the first eigenfunctions of $\Delta_{M}$. A sufficient condition is that $(P_n)$ is quasi-uniform, i.e., that 
\[
\|\dist_M(\cdot,P_n)\|_{\infty}\asymp \min_{1\le i\neq j\le n}\dist_M(x_i^{(n)},x_j^{(n)})\asymp n^{-1/d},
\]
see Filbir and Mhaskar~\cite{FM11}. 

Theorem~\ref{thm:main} improves upon this for $L_q$-approximation in the case $q<p$ by requiring only $\|\dist_M(\cdot,P_n)\|_{L_{\gamma}(M)}\asymp n^{-1/d}$. For integration we require this for $\gamma=sp^*$ with $\frac{1}{p}+\frac{1}{p^*}=1$ which improves upon \cite[Theorem~1]{EGO19}. As mentioned above, Theorem~\ref{thm:main} above builds on its Euclidean counterpart \cite[Theorem~0.1]{KS20}. In general, the strategy of proof for Theorem~\ref{thm:main} is to lift this to manifolds. Important ingredients in the proof are the bounded geometry of the Riemannian manifold $M$, a suitable decomposition of unity and a compatible characterization of the Sobolev space. The latter requires the restriction to the range $1<p<\infty$. We believe, however, that an extension to $p=1,\infty$ is possible. The upper bound of Theorem~\ref{thm:main} is achieved by a linear algorithm, which is based on moving least squares, applied locally to sets where the point set is ``dense enough''. In the case $p>q$, this yields the improvement $\gamma<\infty$ compared to previous works supposing a small covering radius. This enables us to show optimality of other point sets, possibly with rather large holes in them, as long as the ``average hole size'' is of optimal order.

\medskip

In particular, we are interested in the quality of independently and identically distributied (i.i.d.) random sampling points which are distributed uniformly (according to the normalized Riemannian volume measure) on $M$.
We think this is a good and common model
for sampling points that cannot be chosen freely.
By 
Theorem~\ref{thm:main}, it only remains to examine the covering radius
and the distortion of random points.
We refer to Cohort \cite{Coh04} as well as Reznikov and Saff \cite{RS16} where this has been done in similar settings.
We obtain the following.

\begin{prop}\label{pro:randomdistortion}
	Let $X_1,X_2,\ldots$ be i.i.d.\ with respect to the normalized Riemannian volume on $M$ and let $\alpha>0$.
Then
\[
\EE \, \norm{\dist_M(\,\cdot\,,\{X_1,\ldots,X_n\})}_{L_{\gamma}(M)}^\alpha
\, \asymp \, 
\begin{cases}
n^{-\alpha/d} & \text{if } 0<\gamma<\infty,\\
(n/ \log n)^{-\alpha/d} & \text{if } \gamma=\infty,
\end{cases}
\]
where the implied constants are independent of $n$.
\end{prop}

This should be compared with the behavior of optimal sampling points, displayed in \eqref{eq:opt-dist}. Thus, on average, the distortion of random points is asymptotically optimal,
while the covering radius of random points exceeds the covering radius of optimal points
by a logarithmic factor.
We therefore obtain from Theorem~\ref{thm:main} that 
i.i.d.\ uniform sampling points
are, on average, asymptotically optimal for $L_q(M)$-approximation on $H^s_p(M)$
if and only if $q<p$.
In the case $q\ge p$, there is a loss of a logarithmic factor
compared to optimally chosen sampling points. For the integration problem we have asymptotic optimality on average for every $1<p<\infty$ which  closes a logarithmic gap left open in Corollary 1 of~\cite{EGO19}. 

\medskip

As a geometric application of our results, we deduce an asymptotic equivalence for the minimal weighted spherical $L_2$-discrepancy. The concept of discrepancy is central to the study of irregularity of distribution and cubature rules, see e.g. the books \cite{DT97,NW08,NW10} for more information. In particular, in \cite[Section 9.3]{NW10} the connection between weighted $L_2$-discrepancy and integration in reproducing kernel Hilbert spaces is discussed. Further, Matou\v{s}ek \cite{Mat98} used weighted $L_2$-discrepancy to obtain lower bounds on the tractability of high-dimensional integration on $[0,1]^d$. 

\medskip

Consider the $d$-dimensional sphere $\mathbb{S}^d=\{x\in\IR^{d+1}\colon  \|x\|_2=1\}$ equipped with a Riemannian metric inducing the usual geodesic distance $\dist_{\mathbb{S}^d}$. Then geodesic balls are given by spherical caps
\[
C(x,t)
:=\{y\in\mathbb{S}^d\colon  x\cdot y \ge t\} \quad\text{with center }x\in\mathbb{S}^d\text{ and height }t\in [-1,1].
\]
Given points $P=\{x_1,\ldots,x_n\}\subset\mathbb{S}^d$ and weights $w_1,\ldots,w_n\in\IR$ attached to them, their weighted spherical $L_2$-discrepancy is given by
\[
D_2(\{(x_j,w_j)\}_{j=1}^n)
:=\Big(\int_{[-1,1]}\int_{\mathbb{S}^d}\Big\vert \sum_{j=1}^n w_j \mathbf{1}_{C(x,t)}(x_j) - \sigma_d(C(x,t))\Big\vert^2 \dd\sigma_d(x)\dd t\Big)^{1/2},
\]
where $\sigma_d$ is the normalized uniform measure on $\mathbb{S}^d$ and $\mathbf{1}_A$ denotes the characteristic function of a set $A$. It returns the average approximation error when one uses the measure $\sum_{j=1}^n w_j \delta_{x_j}$ to approximate the volume of spherical caps. For equal weights $w_j=1/n$ it reduces to spherical $L_2$-discrepancy which is connected to energy optimization via Stolarsky's principle. We refer to Brauchart and Dick~\cite{BD13} for an elegant proof which uses the Hilbert space structure of the Sobolev space $H^{(d+1)/2}_2(\mathbb{S}^d)$.

\medskip

In particular, the representation of the reproducing kernel of $H^{(d+1)/2}_2(\mathbb{S}^d)$
obtained in~\cite{BD13}  implies for any cubature rule $f\mapsto Q_{P,w}(f):=\sum_{j=1}^n w_j f(x_j)$ the identity 
\begin{equation}\label{eq:scd}
e(Q_{P,w},H^{(d+1)/2}_2(\mathbb{S}^d),\mathrm{INT})
\,=\, D_2(\{(x_j,w_j)\}_{j=1}^n),
\end{equation}
if $H^{(d+1)/2}_2(\mathbb{S}^d)$ is equipped with a suitable equivalent norm,
see Section~\ref{sec:scd} for details.
Thus, Theorem~\ref{thm:main} has the following consequence for the minimal 
weighted spherical $L_2$-discrepancy.
\begin{cor}\label{cor:disc}
	For all $d\in\IN$ and $P=\{x_1,\dots,x_n\}\subset \mathbb{S}^d$, we have 
\[
\inf_{w_1,\ldots,w_n}D_2(\{(x_j,w_j)\}_{j=1}^n)
\,\asymp\, \|\dist_{\mathbb{S}^d}(\,\cdot\,,P)\|_{L_{d+1}(\mathbb{S}^d)}^{(d+1)/2}\,,
\]
where the implied constants are independent of $P$.
\end{cor}

Together with the lower bound \eqref{eq:opt-dist} for the minimal distortion we recover the known lower bound on spherical $L_2$-discrepancy of any $n$-point set which is of order $n^{-1/2-1/2d}$.
By Proposition~\ref{pro:randomdistortion} random points on average achieve this rate if optimal weights are chosen for each realization. 
\begin{cor}\label{cor:disc-ran}
	Let $X_1,X_2,\ldots$ be i.i.d.\ on $\mathbb{S}^d$ distributed according to $\sigma_d$. Then
\[
\IE\,\inf_{w_1,\ldots,w_n}D_2(\{(X_j,w_j)\}_{j=1}^n)
\,\asymp\, n^{-1/2-1/2d}.
\]
\end{cor}

\medskip

We give a brief overview of the remainder of this article. In Section~\ref{sec:prelim} we introduce notation and concepts underlying our results and their proofs. Section~\ref{sec:proof} contains the proof of Theorem~\ref{thm:main} and in Section~\ref{sec:random} we analyze the distortion of random points. In particular, we prove a limit theorem which is stronger and more general than needed and might be of independent interest.

\smallskip

\section{Preliminaries}
\label{sec:prelim}

\medskip

Let $d\in\IN$ and $M=(M,g)$ be a connected compact Riemannian manifold of dimension $d$, which is equipped with a smooth Riemannian metric $g$. Note that $M$ is smooth, without boundary and closed. In the following we present useful information needed later on which we derive from Chavel \cite{Cha93}, Jost \cite{Jos17} and Triebel \cite[Chapter 7.2]{Tri92}. For the sake of conforming to standard notation in this field, we abuse notation by overloading the symbols $p$ and $x$.

\medskip

If $U\subset M$ is open and $x:U\to\RR^d$ is any smooth chart, 
we can identify the tangent space $T_p M$ at $p\in U$ with $\RR^d$ 
by taking the basis $\{\frac{\partial}{\partial x_i}: i=1,\ldots, d\}$ 
and any tangent vector in $T_p M$ may be written as 
$\xi=\sum_{i=1}^d \xi^i \frac{\partial}{\partial x_i}$, $\xi^i\in \RR$. 
In the coordinates given by $x$ the metric tensor $g$ at $p\in U$ can be written as a positive definite matrix,
whose entries are given by 
the inner product of the tangent vectors 
$\frac{\partial}{\partial x_{i}}$ and $\frac{\partial}{\partial x_{j}}$ 
in the tangent space $T_p M$, where $i,j=1,\ldots,d$.
We denote these entries by $g_{ij}^{x}(z)$, where $z=x(p)$.
Then $g_{ij}^{x}: x(U)\to\R$ is a smooth function. 

\subsection{The exponential map} \label{subsec:exp}
Choosing a basis $\{e_1,\ldots,e_d\}$ of the tangent space $T_p M$ at $p\in M$ which is orthonormal with respect to $g$ we can identify $T_p M$ with $\IR^d$ and, abusing notation, we write $\exp_p$ for the exponential map taking $y\in \IR^d$ to $c_v(1)\in M$, where $v=\sum_{i=1}^{d}y_i e_i$ and $c_{v}$ is the unique geodesic with $c_{v}(0)=p$ and tangent vector $v$.  This map is a diffeomorphism defined on a neighbourhood around $0\in \IR^d$ and the supremum of all $\rho>0$ such that it is a diffeomorphism on the Euclidean ball $B(0,\rho):=\{x\in \RR^d: \norm{x}_{2}<\rho\}$ is called the local injectivity radius $r_p$. The (global) injectivity radius is $\inj(M)=\inf_{p\in M} r_p$, which is positive as $M$ is compact. For $0<\rho< \inj(M)$ 
the set $B_M(p,\rho):=\exp_p(B(0,\rho))$ is a geodesic ball with center $p$ and radius $\rho$. Then $\exp_p^{-1}: B_M(p,\rho)\to B(0,\rho)$ is a chart and the corresponding coordinates are called normal coordinates with center~$p$.

\medskip

The Riemannian metric $g$ induces a metric $\dist_{M}$ on $M$, which then becomes a complete metric space.
We want to record the following basic result, which allows us to locally compare the Riemannian distance with the Euclidean distance. It is well known and can be deduced for example from the proof of \cite[Corollary 1.4.1]{Jos17}.
\begin{lemma}\label{lem:expbilip}
Let $p\in M$ and $0<\rho<\inj(M)$. The map $\exp_p: B(0,\rho)\to B_M(p,\rho)$ is bi-Lipschitz. That is, there are constants $0<c<C$ such that
\[
c\|z_1-z_2\|_{2}\,\le\, \dist_M(\exp_p(z_1),\exp_p(z_2))\,\le\, C \|z_1-z_2\|_{2} \qquad \text{for all } z_1,z_2\in B(0,\rho).
\]
\end{lemma}

\subsection{Integration}
Let $U\subset M$ be open and let $x:U\to \IR^d$ be a chart of $M$.
If $z\in x(U)$, we let $g^{x}(z)$ be the determinant of $(g_{ij}^{x}(z))_{i,j=1}^d$.
If $f:M\to\IR$ is supported in $U$ 
and measurable with respect to the Borel $\sigma$-algebra
we define the integral of $f$ by
\begin{equation}\label{def:intsupp}
\int_M f \,\dvol_M:=\int_{x(U)} f(x^{-1}(z))\sqrt{g^{x}(z)}\, \dd z,
\end{equation}
if this quantity exists and is finite. It is independent of the choice of the coordinates~$x$.
We can extend the integral to global functions 
using a partition of unity, which we now describe. 
Since $M$ is compact, there is a finite atlas $\{(U_j,x_j)\}_{j\le J}$ of $M$.
For every such atlas, there exists a subordinate partition of unity $\{(U_j,\psi_j)\}_{j\le J}$ satisfying
\begin{equation}\label{eq:partofone}
	\psi_{j}\in C^{\infty}(M), \quad \mathrm{supp}\, \psi_{j}\subset U_{j}, \quad 0\le \psi_{j}(x) \le 1, \quad \sum_{j=1}^{J}\psi_{j}(x)=1
\end{equation}
for all $j\le J$ and $x\in M$, see, e.g., Tu~\cite[Proposition~13.6]{Tu11}.
Here supp denotes the support of a function. We define the integral of a measurable function $f\colon M\to\IR$ by setting 
\begin{equation}\label{def:int}
\int_M f \,\dvol_M := \sum_{j=1}^J \int_M f\psi_{j}\,\dvol_M 
\end{equation}
if this quantity exists and is finite.
This definition of the integral is independent of the atlas and the partition of unity.
In particular, \eqref{def:intsupp} and \eqref{def:int} yield the same result
if $f$ is supported in $U$.
As usual, the integral over a Borel measurable set $E\subset M$ is defined by
\[
 \int_E f \,\dvol_M := \int_M f\, \mathbf{1}_E \,\dvol_M
\]
and the Riemannian volume of $E$ is defined by 
$\vol_M(E)=\int_M \mathbf{1}_E\,\dvol_M$.
Note that the volume is well-defined and finite for all Borel sets $E\subset M$.

\subsection{Bounded Geometry}

Being compact, $M$ is of bounded geometry in the sense of Triebel \cite{Tri86,Tri92} 
(see also Cheeger, Gromov and Taylor \cite{CGT82} and Shubin~\cite{Shu92}). This means that 
for every $0<\rho<\inj(M)$ and every multiindex $\alpha\in \NN_0^d$
there exist constants $c>0$ and $c_{\alpha}>0$ such that 
\begin{equation}\label{eq:boundedgeometry}
\sqrt{g(z)}\ge c 
\qquad \text{and} \qquad |D^{\alpha}g_{ij}(z)|\le c_{\alpha} \ \text{ for every multiindex }\alpha
\end{equation}
in the normal coordinates of every local chart $(B_M(p,\rho),\exp_p^{-1})$, where $z\in B(0,\rho)$ is arbitrary. From now on, we shall work in normal coordinates only, which justifies the abuse of notation. 
In particular, we find constants 
$c,C>0$ 
such that $c\le \sqrt{g(z)} \le C$ for all $z\in B(0,\rho)$ 
in coordinates of the chart $(B_M(p,\rho),\exp_p^{-1})$. 
This follows from \eqref{eq:boundedgeometry} via $|g_{ij}(z)|\le c_0$ and a determinant bound in terms of entries.  From this one can obtain a useful integral estimate, which we now state.
\begin{lemma}\label{lem:integraltransfer}
There exist constants $c,C>0$ such that for every non-negative function $f$ 
and every $p\in M$ and $0<\delta<\inj(M)/2$
it holds that
\[
c \int_{B(0,\delta)} f(\exp_p(z))\,\dd z
\,\le \int_{B_M(p,\delta)} f(x)\, \dvol_M(x) 
\,\le\, C \int_{B(0,\delta)} f(\exp_p(z))\,\dd z.
\]
\end{lemma}
\smallskip 
\begin{proof}
We use the chart $\exp_p^{-1}\colon B_M(p,2\delta) \to B(0,2\delta)$ and the definition~\eqref{def:intsupp}
of the integral to compute
\begin{multline*}
 \int_{B_M(z,\delta)} f(x)\, \dvol_M(x)
 = \int_{M} (f\,\mathbf{1}_{B_M(p,\delta)})(x)\, \dvol_M(x) \\
= \int_{B(0,2\delta)} (f\,\mathbf{1}_{B_M(p,\delta)})(\exp_{p}(z))\sqrt{g(z)}\, \dd z
= \int_{B(0,\delta)} f(\exp_{p}(z))\sqrt{g(z)}\, \dd z.
\end{multline*}
Using the estimates $c\le \sqrt{g(z)} \le C$ completes the proof.
\end{proof}

\subsection{Sobolev Spaces}
\label{sec:sobolev}

In this subsection we collect useful information about the Bessel potential (Sobolev) spaces $H^{s}_p(M)$, where $s>0$ and $1<p<\infty$. First we give some details regarding their definition and then present an equivalent definition by means of a partition of unity. This will be important for the proof of Theorem~\ref{thm:main}.
Recall that
\[
	H^{s}_{p}(M)
	=\{f\in L_{p}(M):(I-\Delta_M)^{s/2}f\in L_{p}(M)\},
\]
where $I$ is the identity on $L_{p}(M)$ and $\Delta_M$ is the Laplace-Beltrami operator on $M$. The norm of $f\in H^s_p(M)$ is
\[
\norm{f}_{H^{s}_{p}(M)}=\norm{(I-\Delta_M)^{s/2}f}_{L_{p}(M)}.
\]
We shall always assume that $s>d/p$ such that the functions in $H^s_p(M)$ are continuous and function values are well-defined. In addition to \cite{Tri92} see also Strichartz~\cite{Str83} for details. Alternatively, the spaces can be defined using a partition of unity and the corresponding spaces on $\IR^d$. 
\smallskip

Let $\Delta$ be the usual Laplace operator on $\IR^d$. Then the Bessel potential space on $\IR^d$ is given by
\begin{align*}
	H^{s}_{p}(\IR^d)
	&:=\{f\in L_{p}(\IR^d): (I -\Delta)^{s/2}\in L_{p}(\IR^d)\}\\
	&=\{f\in L_{p}(\IR^d):\mathfrak{F}^{-1}( (1+\|\xi\|_2^2)^{s/2}\mathfrak{F}f)\in L_{p}(\IR^d)\},
\end{align*}
where $I$ is the identity on $L_{p}(\IR^d)$, with $\norm{f}_{H^{s}_{p}(\IR^d)}=\norm{\mathfrak{F}^{-1}( (1+\|\xi\|_2^2)^{s/2}\mathfrak{F}f)}_{L_{p}}$ for $f\in H^s_p(\IR^d)$. Here $\mathfrak{F}$ 	($\mathfrak{F}^{-1}$) denotes the (inverse) Fourier transform. 
\medskip

 In order to give the equivalent definition, we follow Triebel \cite[Section 7]{Tri92}, who considers the more general scale of $F$-spaces, also known as Triebel-Lizorkin spaces. Let $0<\delta<\inj(M)/8$ be small enough and consider a finite covering of $M$ by the geodesic balls $B_j:=B_M(p_j,\delta), j\le J,$ as well as a subordinate partition of unity $\Psi=\{(B_j,\psi_j)\colon j\le J\}$ satisfying \eqref{eq:partofone} with $U_j$ replaced by $B_j$. 

 We can now state the equivalent definition.

 \begin{prop}\label{pro:equivalence}
	Let $M$ and $\Psi=\{(B_j,\psi_j)\colon j\le J\}$ be as above. Further, let $s>0$ and $1<p<\infty$. Then the space of all functions $f\in L_p(M)$, for which 
\[
\|f\|_{H^s_p(M)}^{\Psi}:=\Big(\sum_{j=1}^J \|(f\psi_j) \circ \exp_{p_j}\|_{H^s_p(\IR^d)}^p\Big)^{1/p}
\]
is finite, coincides with $H^s_p(M)$ and $\|\cdot\|_{H^s_p(M)}^{\Psi}$ defines an equivalent norm on it. Moreover, if $\Phi$ is another partition of unity as above, the norms $\|\cdot\|_{W^s_p(M)}^{\Psi}$ and $\|\cdot\|_{W^s_p(M)}^{\Phi}$ are equivalent. Here, the exponential map $\exp_{p_j}$ is as in Section~\ref{subsec:exp}.
\end{prop}

\begin{proof}
This follows from \cite[Theorem 7.2.3]{Tri92} and the fact that the spaces $H^s_p$ are contained in the scale of $F$-spaces, see \cite[Theorem 2.5.6]{Tri83} and \cite[Theorem 7.4.5]{Tri92} for this on $\IR^d$ and $M$, respectively. 
\end{proof}

\begin{rem}
	Actually, one can define the more general scales of Triebel-Lizorkin and Besov spaces on $M$ completely analogously as in Proposition~\ref{pro:equivalence}, and as in \cite{KS20}, our results may be obtained for the Triebel-Lizorkin spaces $F^s_{p\tau}(M)$, where $1<p<\infty$ and $0< \tau\le\infty$.
\end{rem}

It follows from Proposition \ref{pro:equivalence} and the corresponding embedding on $\IR^d$ that $H^s_p(M)$ is continuously embedded into $C(M)$, the space of continuous functions on $M$ whenever $s>d/p$.  If $s\in \IN$, then $H^s_p(\IR^d)$ coincides with the Sobolev space $W^s_p(\IR^d)$ of functions with finite norm
\[
\norm{f}_{W^s_p(\IR^d)}
:= \Big( \sum_{\abs{\alpha}\le s} \norm{D^{\alpha}f}_{L_p(\IR^d)}^p \Big)^{1/p},
\]
where $\alpha\in \IN_0^d$ is a multiindex, $|\alpha|=\alpha_1+\dots+\alpha_d$ and $D^{\alpha}=\frac{\partial^{|\alpha|}}{\partial x_1^{\alpha_1}\cdots \partial x_d^{\alpha_d}}f$ denotes a weak partial derivative of order $|\alpha|$. For this fact note that norms on finite-dimensional spaces are equivalent and see, e.g., \cite[Theorem 2.5.6]{Tri83}. Analogously, one can use covariant derivatives to define a Sobolev space of integer smoothness $s\in\IN$ on a manifold and these coincide again with $H^s_p(M)$, see Theorem~7.4.5 in~\cite{Tri92}.
More about the spaces $H^s_p(M)$ may be found in the paper \cite{BCC+14} by Brandolini et al., where the eigensystem of $\Delta_M$ is employed and also $p=1,\infty$ is allowed. We exclude these boundary values of $p$ because of Proposition~\ref{pro:equivalence}.

\subsection{Spherical cap discrepancy}
\label{sec:scd}

As mentioned in the introduction, see~\eqref{eq:scd},
the minimal worst case error of a point set for the integration
problem on $H := H^{(d+1)/2}_2(\mathbb S^d)$
equals the minimal weighted spherical cap discrepancy of the point set.
For completeness, we give some details on this relation while relying on standard techniques, see, e.g.,~Novak and Wo\'zniakowski \cite[Section 9]{NW10}.
\smallskip

According to \cite[Theorem~1]{BD13},
an equivalent norm on the Hilbert space $H$ is given via the reproducing kernel
\[
 K\colon \mathbb S^d \times \mathbb S^d \to \IR,
 \quad
 K(x,y) \,=\, \int_{[-1,1]} \int_{\mathbb S^d}\, \mathbf{1}_{C(z,t)}(x) \mathbf{1}_{C(z,t)}(y)\,\dd\sigma_d(z)\dd t.
\]
 The representer $h\in H$ of the integral on $H$ is given by
\[
 h(x) \,=\, \langle h, K(x,\cdot) \rangle \,=\, \int_{\mathbb S^d}\, K(x,y)\,\dd\sigma_d(y).
\]
Thus the worst case error of the cubature rule $f\mapsto Q(f):=\sum_{j=1}^n w_j f(x_j)$
can be computed as
\begin{multline*}
 e(Q,H,\mathrm{INT})^2
 \,=\, \sup_{\Vert f \Vert_H\le 1} \left| \int_{\mathbb S^d}\, f(x)\,\dd\sigma_d(x) - \sum_{j=1}^n w_j f(x_j) \right|^2 
 \,=\, \left\Vert h - \sum_{j=1}^n w_j K(x_j,\cdot) \right\Vert^2 \\
 \,=\, \int_{\mathbb S^d}\int_{\mathbb S^d} K(x,y)\,\dd\sigma_d(y)\dd\sigma_d(x) 
 - 2 \int_{\mathbb S^d} \sum_{j=1}^n w_j K(x_j,y)\,\dd\sigma_d(y)
 + \sum_{i,j=1}^n w_i w_j K(x_i,x_j).
\end{multline*}
Inserting the definition of the kernel $K$ and interchanging integrals, we arrive at
\begin{multline*}
 =\,\int_{[-1,1]} \int_{\mathbb S^d}\bigg[ \sigma_d(C(z,t))^2
 - 2 \sum_{j=1}^n w_j \mathbf{1}_{C(z,t)}(x_j)\,\sigma_d(C(z,t)) \\
 \qquad + \sum_{i,j=1}^n w_i w_j  \mathbf{1}_{C(z,t)}(x_i) \mathbf{1}_{C(z,t)}(x_j) 
  \bigg] \dd\sigma_d(z)\dd t
 \,=\, D_2(\{(x_j,w_j)\}_{j=1}^n).
\end{multline*}

\section{The proof of Theorem~\ref{thm:main}}
\label{sec:proof}

Thanks to our preparations, we are able to lift Theorem~\ref{thm:main}
from Euclidean balls to Riemannian manifolds.
More precisely, we use the following proposition from \cite{KS20}.
All parameters are given as in Theorem~\ref{thm:main}.

\begin{prop}[\cite{KS20}]\label{pro:maineuclidean}
Let $\delta>0$ and set $B=B(0,\delta)\subset \IR^d$. 
Then there are positive constants $C$ and $c$ such that the following holds
for all finite and nonempty point sets $P\subset B$.
\begin{enumerate}
 \item There is a non-negative function $f\in H_p^s(\R^d)$ with support in $B$ 
 and $f\vert_P=0$ such that $\Vert f \Vert_{H_p^s(\R^d)} \le 1$ and
 \[
  \Vert f \Vert_{L_q(B)} \,\ge\, c\, \norm{\dist(\cdot,P) }_{L_\gamma(B)}^\alpha.
 \]
 \item There are bounded functions $u_x\colon B\to \R$ for $x\in P$ 
 such that
 for all $f\in H_p^s(\R^d)$ with support in $B$ and $\Vert f \Vert_{H_p^s(\R^d)} \le 1$ we have 
 \[
  \bigg\Vert f - \sum_{x\in P} f(x) u_x  \bigg\Vert_{L_q(B)}
  \,\le\, C\, \norm{ \dist(\cdot,P\cup \partial B) }_{L_\gamma(B)}^\alpha.
 \]
\end{enumerate}
\end{prop}

\smallskip

\begin{proof}
	The first property is proven precisely this way in \cite{KS20} and thus we turn to the second property.
Let $Q\subset \partial B$ be finite.
By Theorems~1 and 2 
in \cite{KS20} there are bounded functions $u_x\colon B\to\R$ for $x\in P\cup Q$ such that 
we have for all $f\in H_p^s(B)$ with $\Vert f \Vert_{H_p^s(B)} \le 1$
that 
\[
 \bigg\Vert f - \sum_{x\in P\cup Q} f(x) u_x  \bigg\Vert_{L_q(B)}
  \le\, C\, \norm{ \dist(\cdot,P\cup Q) }_{L_\gamma(B)}^\alpha
\]
with a constant $C$ independent of $P$ and $Q$ (and $f$). See also \cite[Remark~8]{KS20}.
If $f$ is supported in $B$, this turns into
\[
 \bigg\Vert f - \sum_{x\in P} f(x) u_x  \bigg\Vert_{L_q(B)}
  \le\, C\, \norm{ \dist(\cdot,P\cup Q) }_{L_\gamma(B)}^\alpha
\]
Taking the infimum over all finite point sets $Q\subset \partial B$, we get the statement.
More precisely, we find a sequence of finite point sets $Q_k \subset \partial B$
such that the $1/k$-balls centered at $Q_k$ cover $\partial B$.
Then $\dist(\cdot,P\cup Q_k)$ converges uniformly to $\dist(\cdot,P\cup \partial B)$ on $B$ 
and thus $\norm{ \dist(\cdot,P\cup Q_k) }_{L_\gamma(B)}^\alpha$ converges
to $\norm{ \dist(\cdot,P\cup \partial B) }_{L_\gamma(B)}^\alpha$. 
\end{proof}

\smallskip

In the following, fix a partition of unity $\Psi=\{(B_j,\psi_j) \colon j\le J\}$ 
as in Proposition~\ref{pro:equivalence}, 
where $B_j=B_M(p_j,\delta)$ with $0<\delta<\inj(M)/24$ small enough to accommodate for a factor of three, see the proof of the lower bound below for a justification. 
We use the shorthand $B=B(0,\delta)$ and $\exp_j=\exp_{p_j}$.

\subsection{The upper bound}

We construct our algorithm as follows.  
Let $f\in H^s_p(M)$ with $\|f\|_{H^s_p(M)}\le 1$ and $j\le J$ be arbitrary but fixed for now.
We write 
\[
f_j:=(\psi_j f)\circ \exp_{j}
\]
which
is supported in $B$ and belongs to $H^s_p(\IR^d)$ with 
\[
\|f_j\|_{H^s_p(\IR^d)}\le \|f\|_{H^s_p(M)}^{\Psi}\le C_{\Psi}\|f\|_{H^s_p(M)}\le C_{\Psi},
\]
where $C_{\Psi}$ is independent of $f$. This follows from Proposition~\ref{pro:equivalence}.
We take the point set $P_j:= \exp_{j}^{-1}(P\cap B_j)\subset B$. 
Applying Proposition~\ref{pro:maineuclidean} to the rescaled functions $C_{\Psi}^{-1}f_j$ we get bounded functions $u_x^{(j)}\colon B \to \IR$, $x\in P_j$,
such that we have for all $f\in H^s_p(M)$ that
\begin{equation}\label{eq:localestimate0}
  \Bnorm{ f_j - \sum_{x\in P_j} f_j(x) u_x^{(j)}  }_{L_q(B)}
	\,\le\, C_{\Psi}C\, \norm{ \dist(\cdot,P_j\cup\partial B) }_{L_\gamma(B)}^\alpha.
\end{equation}
Define a function $A_j(f)$ on $B_j$ by
\[
A_j(f) := \Big(\sum_{x\in P_j} f_j(x)\, u_x^{(j)} \Big) \circ \exp_{j}^{-1}
= \sum_{p\in P\cap B_j} (\psi_j f)(p)\, u_p^{(j)}  
\]
and extend the function to $M$ by zero. 
Here we carried out the substitution $p=\exp_{j}(x)$ 
and denote $u_p^{(j)}=u_x^{(j)}\circ\exp_{j}^{-1}$.
Rewriting \eqref{eq:localestimate0} gives
\begin{equation} \label{eq:localestimate}
 \norm{ (\psi_j f - A_j(f))\circ \exp_{j}  \,}_{L_q(B)}
  \,\le\, C\, \norm{ \dist(\cdot,P_j\cup\partial B) }_{L_\gamma(B)}^\alpha.
\end{equation}
We proceed like this for 
every $j\le J$. 
Note that the constant $C$ is independent of $f$, $P$ and $j$. 
We set
\[
A(f) := \sum_{j=1}^J A_j(f).
\]
The recovery operator $A\colon H_p^s(M) \to L_\infty(M)$ is linear and uses only function values at $P$. 
It remains to exhibit an appropriate upper bound. This is now straightforward. The triangle inequality and $\sum_{j=1}^{J}\psi_j\equiv 1$ imply
\begin{align*}
\norm{f - A(f)}_{L_q(M)} 
\le \sum_{j=1}^J\norm{\psi_j f - A_j(f)}_{L_q(M)}.
\end{align*}
We apply the upper estimate of Lemma \ref{lem:integraltransfer} for every $j$ to the non-negative functions $|\psi_j f - A_j(f)|^q,$ supported in $B_j=B_M(z_j,\delta)$, and arrive at
\[
\norm{f - A(f)}_{L_q(M)} \le C^{1/q} \sum_{j=1}^J \norm{ (\psi_j f - A_j(f) )\circ \exp_{j}\,}_{L_q(B)}.
\]

Applying the local estimate \eqref{eq:localestimate} to each term in the sum yields 
\[
\norm{f - A(f)}_{L_q(M)} \lesssim  \sum_{j=1}^J  \big\Vert \dist(\cdot,P_j\cup\partial B) \big\Vert_{L_\gamma(B)}^\alpha.
\]
Let us note that the constants of Lemma~\ref{lem:expbilip} and Lemma~\ref{lem:integraltransfer} are independent of $p$, but we do not need this fact here and in the following, since we use the lemmas only for a constant number of points $p_j$.
Here and in the following, the notation $A \lesssim B$ means that there exists an implicit constant $C>0$, which does not depend on $f$ or $P$, such that $A\le CB$ everywhere. We interpret $A \gtrsim B$ as $B\lesssim A$.

Using the lower estimate of Lemma \ref{lem:integraltransfer} for 
$\dist(\exp_{j}^{-1}(\cdot),P_j\cup\partial B)^{\gamma}$ 
for every $j$ gives
\begin{equation}\label{eq:upper-dist}
\norm{f - A(f)}_{L_q(M)} \,\ls\, 
\sum_{j=1}^J \norm{ \dist(\exp_{j}^{-1}(\cdot),P_j\cup\partial B)}_{L_\gamma(B_j)}^\alpha.
\end{equation}
Note that $\exp_j(P_j\cup\partial B) = (P\cap B_j) \cup \partial B_j$
and Lemma \ref{lem:expbilip} implies that
\[
\dist(\exp_{j}^{-1}(x),P_j\cup\partial B)
\le c^{-1}\dist_M(x,(P\cap B_j) \cup \partial B_j) 
\le c^{-1}\dist_M(x,P)
\] 
for some $c>0$ and all $x\in B_j$.
Combining these inequalities and taking the $L_{\gamma}(B_j)$-norm we derive from \eqref{eq:upper-dist} that
\[
\norm{f - A(f)}_{L_q(M)} \,\ls\, \sum_{j=1}^J  \norm{ \dist_M(\cdot,P)}_{L_\gamma(B_j)}^\alpha \le J \norm{ \dist_M(\cdot,P)}_{L_\gamma(M)}^\alpha.
\]
 This proves the upper bound for Theorem~\ref{thm:main}, where we allow the implicit constants to depend on the chosen partition, and in particular on $J$.
The upper bound in Theorem~\ref{thm:main}(b) is achieved by the cubature rule $Q(f):=\int_M A(f)\,\dvol_M$ which follows from
\[
\Big\vert \int_{M} f\, \dvol_M - \int_M A(f)\, \dvol_M \Big\vert
\,\le\, \|f-A(f)\|_{L_1(M)}.
\]

\subsection{The lower bound}

As in \cite{KS20},
the proof of the lower bound relies on 
the well-known technique of fooling functions.
Namely, we provide a non-negative function $f_*$ 
which belongs to the unit ball of $H_p^s(M)$
and vanishes on the point set $P$.
Then any algorithm $A=\varphi\circ N_P$ based on $P$ cannot distinguish $f_*$ from its negative, i.e., it satisfies $A(f_*)=A(-f_*)$. Thus,
	\begin{equation*}
	\begin{split}
	 e(P,H_p^s(M),& L_q(M)) 
	 \,=\, \inf_{A=\varphi\circ N_P} \sup_{\|f\|_{H^s_p(M)}\le 1} \Vert f-A(f) \Vert_{L_q(M)}\\
	 &\ge\, \inf_{u\in L_q(M)}\, \max\left\{ \Vert - f_* + u \Vert_{L_q(M)}, \Vert f_* + u \Vert_{L_q(M)} \right\}
	  \ge\, \|f_*\|_{L_q(M)},
\end{split}
\end{equation*}
and analogously for the integration problem.  It thus suffices to find $f_*$ with
\begin{equation}\label{eq:fofu}
 \Vert f_*\Vert_{L_q(M)} \,\gtrsim\, \left\|\dist_M(\,\cdot\,, P)\right\|_{L_{\gamma}(M)}^\alpha.
\end{equation}

In order to prepare the proof, we introduce a set of nice partitions of unity, which are well adapted for constructing our fooling function. 
Namely, for each $k\le J$ we choose a partition of unity $\Psi_k$ by removing all elements of the partition $\Psi$ which intersect the ball $B_k$ and adjoin the larger ball $3B_k$ together with the sum of the deleted functions. In this way, we obtain a partition with only one element not equal to the zero function on $B_k$.
More precisely, let $J_k:=\{j\le J: B_j\cap B_k = \emptyset\}$ and 
\[
\Psi_k:= \{(3B_j,\psi_j) \colon j\in J_k\} \cup \{(3B_k,\widetilde{\psi}_k)\} \quad \text{with} \quad \widetilde{\psi}_k=\sum_{j\in J_k^C} \psi_j, 
\]
where $J_k^C=\{j\le J: B_j \cap B_k\neq \emptyset\}$.
It is readily checked that we have a partition of unity as before: Any point which belonged to some $B_j, j\in J_k^C,$ is included in $3B_k$ and $\widetilde{\psi}_k$ is supported in $3B_k$. Moreover, we gain the property $\widetilde{\psi}_k(x)=1$ for $x\in B_k$. The enlargement of the balls $B_j$, $j\in J_k,$ to $3B_j$ is due to technical reasons since only equally sized balls are used in \cite{Tri92}.
As the Sobolev norms $\norm{\cdot}_{H^s_p(M)}^{\Psi_{k_1}}$ and $\norm{\cdot}_{H^s_p(M)}^{\Psi_{k_2}}$ are equivalent for all $k_1,k_2\le J$, we can take suitable constants $c_{\Psi},C_{\Psi}>0$ such that for all $f\in H^s_p(M)$ and all $k_1,k_2\le J$, we have
\[
c_{\Psi}\norm{f}_{H^s_p(M)}^{\Psi_{k_1}}\le \norm{f}_{H^s_p(M)}^{\Psi_{k_2}} \le C_{\Psi}\norm{f}_{H^s_p(M)}^{\Psi_{k_1}}.
\]

We begin to construct our fooling function. If $\gamma<\infty$, then because of
\begin{align*}
\|\dist_M(\cdot, P)\|_{L_{\gamma}(M)}^{\gamma}
&=\sum_{j=1}^J \int_M \psi_j(x) \dist_M(x, P)^{\gamma}\,\dvol_M(x)\\
&\le \sum_{j=1}^J \int_{B_j} \dist_M(x, P)^{\gamma}\,\dvol_M(x),  
\end{align*}
there must exist $i_0\le J$ such that
\begin{equation}\label{eq:badball}
\int_{B_{i_0}} \dist_M(x, P)^{\gamma}\,\dvol_M(x) \,\ge\, \frac{1}{J}\, \|\dist_M(\cdot, P)\|_{L_{\gamma}(M)}^{\gamma}.
\end{equation}
If $\gamma=\infty$, we replace formula \eqref{eq:badball} by
\[
\|\dist_M(\cdot,P)\|_{L_{\infty}(B_{i_0})}
= \|\dist_M(\cdot,P)\|_{L_{\infty}(M)}.
\]
We consider now the partition of unity $\Psi_{i_0}$. Recall that $3B_{i_0}$ belongs to it and $\widetilde{\psi}_{i_0}=1$ on $B_{i_0}$. The exponential maps $\exp_j$ are now considered as mappings from $3B$ to $3B_j$ for $j\in J_k$.

\smallskip

We construct a fooling function on $B$ on which we consider the point set $P_{i_0}=\exp_{i_0}^{-1}(P\cap B_{i_0})$. By Proposition~\ref{pro:maineuclidean} we get a non-negative function $f\in H^s_p(\IR^d)$ with support in $B$, 
\[
\norm{f}_{H^s_p(\IR^d)}\le 1,\quad f|_{P_{i_0}}=0\quad \text{and}\quad\|f\|_{L_q(B)}\gs \|\dist(\cdot, P_{i_0})\|_{L_{\gamma}(B)}^\alpha.
\]
We define the function $f_{\ast}$ by $f_{\ast}(x)=f(\exp^{-1}_{i_0}(x))$ if $x\in B_{i_0}$ and zero else. 

\smallskip

We compute that $\psi_j f_{\ast}\circ \exp_{j}=0$ if $j\in J_{i_0}$, since $\psi_j\equiv 0$ on $B_{i_0}$ for those $j$, and $\widetilde{\psi}_{i_0} f_{\ast} \circ \exp_{i_0}=f$. By definition
\[
\norm{f_{\ast}}_{H^s_p(M)}^{\Psi_{i_0}}
=\Big(\sum_{j\in J_{i_0}} \|\psi_j f_{\ast}\circ \exp_{j}\|_{H^s_p(\RR^d)}^p+\|\widetilde{\psi}_{i_0} f_{\ast}\circ \exp_{i_0}\|_{H^s_p(\RR^d)}^p\Big)^{1/p}
\]
and thus, by Proposition \ref{pro:equivalence},
\[
\norm{f_{\ast}}_{H^s_p(M)}
\,\asymp\,\norm{f_{\ast}}_{H^s_p(M)}^{\Psi_{i_0}}
=\norm{\widetilde{\psi}_{i_0} f_{\ast}\circ \exp_{i_0}}_{H^s_p(\RR^d)}
=\norm{f}_{H^s_p(\RR^d)}
\le 1.
\]
Therefore, we can replace $f_{\ast}$ by a rescaled version such that $\|f_{\ast}\|_{H^s_p(M)}\le 1$. Additionally, we have $f_{\ast}|_P=0$. In the case $q<\infty$, we apply Lemma \ref{lem:integraltransfer} to the support $B_{i_0}$ to get
\[
\|f_*\|_{L_q(M)}
 \,\gs\,   \|f\|_{L_q(B)} 
\,\gs\,  \|\dist(\cdot, P_{i_0})\|_{L_{\gamma}(B)}^{\alpha}.
\]
It is straightforward that 
this relation also holds for $q=\infty$.
It remains to
to prove a lower estimate of the form
\begin{equation}\label{eq:distnormsgoal}
\|\dist(\cdot, P_{i_0})\|_{L_{\gamma}(B)} 
\,\gs\, \norm{\dist_M(\cdot, P)}_{L_{\gamma}(B_{i_0})},
\end{equation}
which by means of \eqref{eq:badball} yields that the function $f_{\ast}$ satisfies
\eqref{eq:fofu} and concludes the proof of the lower bounds of Theorem~\ref{thm:main}.

\smallskip

We start with the right-hand side of \eqref{eq:distnormsgoal}. By Lemma~\ref{lem:integraltransfer} we have for $\gamma<\infty$
\[
\norm{\dist_M(\cdot, P)}_{L_{\gamma}(B_{i_0})}
\,\ls\, \norm{\dist(\exp_{i_0}(\cdot), P)}_{L_{\gamma}(B)}.
\]
For $\gamma=\infty$, we even have equality.
We estimate for $z\in B$ using Lemma~\ref{lem:expbilip}
\[
\dist_M(\exp_{i_0}(z),P)
\le \min_{x\in P\cap B_{i_0}} \dist_M(\exp_{i_0}(z),x)
\,\ls\, \min_{y\in P_{i_0}} \|z-y\|_2
= \dist(z,P_{i_0}).
\]

Applying the $L_{\gamma}(B)$-norm gives
\[
\norm{\dist_M(\exp_{i_0}(\cdot), P)}_{L_{\gamma}(B)}
\,\ls\, \norm{ \dist(\cdot,P_{i_0})}_{L_{\gamma}(B)}, 
\]
and we arrive at \eqref{eq:distnormsgoal}.

\section{Random points}
\label{sec:random}

This section contains the proof of Proposition \ref{pro:randomdistortion} and related results.
That is, we now study the moments of the random quantity $\norm{\dist_M(\,\cdot\,,\{X_1,\ldots,X_n\})}_{L_{\gamma}(M)}$
for identically and uniformly distributed $X_1,X_2,\ldots$ on $M$.
\medskip

For $\gamma=\infty$ the upper and lower bound from Proposition~\ref{pro:randomdistortion} follow from Theorems 2.1 and 2.2 in Reznikov and Saff \cite{RS16} and the Ahlfors-regularity of $M$, see  Lemma~\ref{lem:ahlfors} below.
For $0<\gamma<\infty$ the lower bound follows from relation~\eqref{eq:opt-dist} since random points can be no better than optimal points. We now proceed with the upper bound in the case $0<\gamma<\infty$.
In fact, we prove the following more general and more precise statement using a modification of the arguments in Cohort \cite[Theorem 1]{Coh04}. 

\begin{thm}\label{thm:distortion}
Let $M$ be 
a compact Riemannian manifold of dimension $d\in\NN$
and $\gamma\in (0,\infty)$. 
Let $\nu$ be a Borel probability measure on $M$ 
which is absolutely continuous with respect to $\vol_M$ 
with density $h_{\nu}:M\to [c,\infty)$ for some $c>0$. 
Let $X_1,X_2,\ldots$ be i.i.d. with law $\nu$. 
Then the distortion of the random point set $P_n=\{X_1,\ldots,X_n\}$ satisfies
\[
n^{\gamma/d}\int_M \dist_M(x,P_n)^{\gamma}\,\dvol_M(x)
\ \longrightarrow \ \omega_d^{-\gamma/d}\,\Gamma\Big(1+\frac{\gamma}{d}\Big)\int_M h_{\nu}(x)^{-\gamma/d}\, \dvol_M(x),
\]
as $n\to\infty$ in $L_p(\Omega)$ for all $p \in (0,\infty)$, where $\Omega$ is the underlying probability space. 
\end{thm}

To the best of our knowledge, there are no exact asymptotic results for quantization using random points on Riemannian manifolds so far and we believe that it could be worthwhile to extend other results on random quantization to manifolds. See the book \cite{GL00} by Graf and Luschgy for more information on the Euclidean setting.

\smallskip

The proof of Proposition~\ref{pro:randomdistortion} is complete if we apply Theorem~\ref{thm:distortion} for $p=\alpha/\gamma$. Thus, it remains to prove Theorem \ref{thm:distortion}. The only essential thing we shall use is the following lemma, a proof of which can be found for example in Bl\"umlinger \cite[Lemma 2]{Blu90}.

\begin{lemma}\label{lem:volumeofballs}
Let $M$ be a compact $d$-dimensional Riemannian manifold. Then 
\[
\lim_{r\to 0} \frac{\vol_M(B_M(p,r))}{\vol(B(0,r))}=1
\]
uniformly in $p\in M$, where $\vol$ denotes the $d$-dimensional Lebesgue measure.
\end{lemma}

In other words, letting $\omega_d$ be the $d$-dimensional volume of the unit ball of $\IR^d$ we have $\lim_{r\to 0} \vol_M(B_M(p,r))r^{-d}=\omega_d$ uniformly in $p\in M$.  As a consequence of Lemma~\ref{lem:volumeofballs}, one can obtain the well-known fact that the volume and the metric on $M$ satisfy a compatibility condition which is also known as Ahlfors-regularity.
\begin{lemma}
	\label{lem:ahlfors}
Let $M$ be a compact $d$-dimensional Riemannian manifold. Then there are constants $c,C>0$ such that 
\[
	c\, r^d
	\le \vol_M(B_M(p,r)) 
	\le C\, r^d\quad\text{for all }0<r<\mathrm{diam}(M)\text{ and all }p\in M.
\]
\end{lemma}

We are now ready to prove the limit theorem for the distortion of i.i.d.\ random points on a manifold. 

\begin{proof}[Proof of Theorem~\ref{thm:distortion}]
To ease the notation in this proof, we denote the geodesic distance by $\dist$ and the volume integral on $M$ by $\dd x$. We set 
\begin{align*}
Z_n \,&:=\, n^{\gamma/d} \int_M \dist(x,P_n)^{\gamma}\,\dd x, \\
z \,&:=\, \omega_d^{-\gamma/d}\,\Gamma\Big(1+\frac{\gamma}{d}\Big)\int_M h_{\nu}(x)^{-\gamma/d} \dd x,
\end{align*}
and show that 
\begin{equation}\label{eq:moments}
 \IE\, Z_n^p \to z^p \quad \text{for all } \ p\in\IN.
\end{equation}
From this it follows that, for any even integer $p$,
\[
 \IE\, \vert Z_n - z \vert^p \,=\, \sum_{k=0}^p \binom{p}{k} (-z)^{p-k}\, \IE\,Z_n^k
 \ \xrightarrow{n\to\infty}\ \sum_{k=0}^p \binom{p}{k} (-z)^{p-k} z^k \,=\, 0,
\]
i.e., that $Z_n$ converges to $z$ in $L_p(\Omega)$ for all even integers $p$.
By Jensen's inequality, the convergence holds for all $p\in(0,\infty)$, which completes the proof.

\smallskip

It remains to show \eqref{eq:moments}. Let $p\in \IN$. We have
\[
\EE\, Z_n^p 
\,=\, \EE \int_{(M)^p} \prod_{j=1}^p n^{\gamma/d}
 \dist(x_j,P_n)^{\gamma}\,
\dd x_1\cdots\dd x_p.
\]
Interchanging the integrals yields
\[
\EE\, Z_n^p 
\,=\, 
\int_{(M)^p} \EE\prod_{j=1}^p
n^{\gamma/d}\dist(x_j,P_n)^{\gamma}\,
\dd x_1\cdots\dd x_p.
\]
We use $a_j=\int_{\RR_+} \mathbf{1}_{[v_j,\infty)}(a_j)\,\dd v_j$
for $a_j=n^{\gamma/d}\dist(x_j,P_n)^{\gamma}$ 
and put
\[
 {\rm I}_n(v_1,\ldots,v_p,x_1,\ldots,x_p)
 \,:=\,\EE \, \prod_{j=1}^p \bfone_{[v_j,\infty)}(n^{\gamma/d}\dist(x_j,P_n)^{\gamma})
\]
to see after an interchange of integrals that
\begin{equation*}
\EE\, Z_n^p 
\,=\, 
\int_{(M)^p}\int_{(\RR_+)^p} {\rm I}_n(v_1,\ldots,v_p,x_1,\ldots,x_p)\, \dd v_1\cdots\dd v_p\, \dd x_1\cdots\dd x_p.
\end{equation*}
Since
\begin{equation*}
z^p \,=\,\int_{(M)^p}\int_{(\RR_+)^p} \exp\Big(-\omega_d\sum_{j=1}^p v_j^{d/\gamma}h_{\nu}(x_j)\Big)\, \dd v_1\cdots\dd v_p\, \dd x_1\cdots\dd x_p,
\end{equation*}
the desired statement $\EE\, Z_n^p\to z^p$ follows from the dominated convergence theorem
once we show that
\begin{equation}\label{eq:asconv}
 {\rm I}_n(v_1,\ldots,v_p,x_1,\ldots,x_p)\,\xrightarrow{n\to\infty}\, \exp\Big(-\omega_d\sum_{j=1}^p v_j^{d/\gamma}h_{\nu}(x_j)\Big)
\end{equation}
for almost all $v_1,\ldots,v_p,x_1,\ldots,x_p$
and find an integrable majorant of $I_n$.  We observe that
\begin{multline*}
{\rm I}_n(v_1,\ldots,v_p,x_1,\ldots,x_p)
=\, \PP\left[n^{\gamma/d} \dist_M(x_j,P_n)^{\gamma}\ge v_j \text{ for all }j\right] \\
\,=\, \PP\left[X_i\not\in B_M(x_j, v_j^{1/\gamma}n^{-1/d}) \text{ for all }i,j\right] 
=\, \bigg(1-\nu\Big(\bigcup_{j=1}^p B_M\big(x_j,v_j^{1/\gamma}n^{-1/d}\big)\Big)\bigg)^n.
\end{multline*}

In order to find an integrable majorant of $I_n$, we note that $I_n=0$
whenever there is some $j\le p$ with $v_j^{1/\gamma}n^{-1/d} \ge \mathrm{diam}(M)$.
If there is no such $j$, 
then we use Lemma~\ref{lem:ahlfors} and the lower bound on the density of $\nu$ to find a constant $c_1>0$ such that $\nu(B_M(x,r)) \ge c_1 r^d$ for all $x\in M$ and $0<r<\mathrm{diam}(M)$.
Therefore,
\begin{multline*}
{\rm I}_n(v_1,\ldots,v_p,x_1,\ldots,x_p)
\,\le\, \bigg(1- \max_{j\le p}\, \nu\big(B_M(x_j,v_j^{1/\gamma}n^{-1/d})\big)\bigg)^n \\
\le\, \bigg(1- \frac{c_1 \max_{j\le p}\,v_j^{d/\gamma}}{n} \bigg)^n
\,\le\,  \exp\left( - c_1 \max_{j\le p}\, v_j^{d/\gamma} \right)
\,\le\,  \exp\bigg( - \frac{c_1}{p} \sum_{j=1}^p\, v_j^{d/\gamma} \bigg)
\end{multline*}
which gives an integrable majorant of $I_n$.

\smallskip

For showing the almost sure convergence \eqref{eq:asconv} (and thus completing the proof of \eqref{eq:moments}), we use that for
almost every choice of the $x_j$ and $v_j$ 
we have for $n$ large enough that the balls 
$ B_M(x_j,v_j^{1/\gamma}n^{-1/d})$ are disjoint 
and thus
\[
{\rm I}_n(v_1,\ldots,v_p,x_1,\ldots,x_p) \,=\, \bigg(1-\sum_{j=1}^p \nu\big(B_M(x_j,v_j^{1/\gamma}n^{-1/d})\big)\bigg)^n.
\]
Define $a_n:=\sum_{j=1}^p \nu(B_M(x_j,v_j^{1/\gamma}n^{-1/d}))$ for which we want to show that $(1-a_n)^n \to {\rm e}^{-a}$ for $a:=\omega_d\sum_{j=1}^p v_j^{d/\gamma}h_{\nu}(x_j)$. It is sufficient to show $n a_n\to a$.

\smallskip

To show $na_n\to a$, write, for $1\le j\le p$,
\[
n\nu(B_M(x_j,v_j^{1/\gamma}n^{-1/d}))=\frac{n\vol_M(B_M(x_j,v_j^{1/\gamma}n^{-1/d}))}{\vol_M(B_M(x_j,v_j^{1/\gamma}n^{-1/d}))}\int_{B_M(x_j,v_j^{1/\gamma}n^{-1/d})} h_{\nu}(x)\,\dd x,
\]
where $h_{\nu}\in L_1(M)$ is the density of $\nu$. By Lemma~\ref{lem:volumeofballs} the numerator satisfies
\[
n\vol_M(B_M(x_j,v_j^{1/\gamma}n^{-1/d}))\,\to\, \omega_d v_j^{d/\gamma} 
\quad \text{as } \ n\to\infty,
\]
and by the differentiation theorem of Lebesgue we have
\[
\frac{1}{\vol_M(B_M(x_j,v_j^{1/\gamma}n^{-1/d}))}\int_{B_M(x_j,v_j^{1/\gamma}n^{-1/d})} h_{\nu}(x)\,\dd x 
\,\to\, h_{\nu}(x_j)
\]
for almost all $x_j$ and $v_j$.
Therefore, 
\[
n\nu(B_M(x_j,v_j^{1/\gamma}n^{-1/d}))\to \omega_d v_j^{d/\gamma} h_{\nu}(x_j)
\]
and the desired convergence $n a_n\to a$ follows. This completes the proof of \eqref{eq:asconv} and thus of Theorem~\ref{thm:distortion}.
\end{proof}

\medskip

\begin{rem}
	In the case of $\gamma<\infty$ and the sphere $\mathbb{S}^{d}$ one can prove the upper bound of Proposition~\ref{pro:randomdistortion} with the help of a covering by holes in the random point set and Theorem~2.2 in \cite{BRS+18} due to Brauchart et al.\ which states that the average hole size behaves well.
\end{rem}

\bigskip

\subsection*{Acknowledgement}
We thank Aicke Hinrichs for initiating the collaboration 
of the two authors on this project and both Dmitriy Bilyk and Martin Ehler for interesting discussions on the topic. 
Both authors are supported by the Austrian Science Fund (FWF) Project F5513-N26, which is a part of the Special Research Program \emph{Quasi-Monte Carlo Methods:~Theory and Applications}. This research was funded in whole, or in part, by the Austrian Science Fund (FWF), Project P34808. For the purpose of open access, the authors have applied a CC BY public copyright license to any Author Accepted Manuscript arising from this submission.
\bigskip

\bibliographystyle{plain}
\bibliography{2020-RIS}

\end{document}